%% file: pscc2022.tex
\documentclass{IEEEtran4PSCC}
\ifCLASSINFOpdf
   \usepackage[pdftex]{graphicx}
\else
   \usepackage[dvips]{graphicx}
\fi
\usepackage[cmex10]{amsmath}
\usepackage{amssymb}
\usepackage{amsthm}
\usepackage{flushend}

\usepackage{ifthen}
\newboolean{arxiv}

\newtheorem{theorem}{Theorem}
\newtheorem{lemma}{Lemma}
\newtheorem{definition}{Definition}

\newtheorem{proposition}{Proposition}

\newtheorem{remark}{Remark}

\DeclareMathOperator{\Rainflow}{Rainflow}

\usepackage{subcaption}
\usepackage[linesnumbered,ruled,vlined]{algorithm2e} 
\SetKwFunction{Farrange}{arrange}
\SetKwFunction{Fadd}{add}
\SetKwProg{Fn}{Function}{:}{}

\SetKwInput{KwInput}{Input}                
\SetKwInput{KwOutput}{Output}
\SetKwInput{KwInitialize}{Initialize}

\input{edits_pcyou}
\input{edits}

\usepackage[noadjust]{cite}
\usepackage[linesnumbered,ruled,vlined]{algorithm2e}

\makeatletter
\let\old@ps@headings\ps@headings
\let\old@ps@IEEEtitlepagestyle\ps@IEEEtitlepagestyle
\def\psccfooter#1{%
    \def\ps@headings{%
        \old@ps@headings%
        \def\@oddfoot{\strut\hfill#1\hfill\strut}%
        \def\@evenfoot{\strut\hfill#1\hfill\strut}%
    }%
    \def\ps@IEEEtitlepagestyle{%
        \old@ps@IEEEtitlepagestyle%
        \def\@oddfoot{\strut\hfill#1\hfill\strut}%
        \def\@evenfoot{\strut\hfill#1\hfill\strut}%
    }%
    \ps@headings%
}
\makeatother

\IEEEoverridecommandlockouts

\begin{document}

\setboolean{arxiv}{true}  

%
\title{A Market Mechanism for Truthful Bidding with Energy Storage}


\author{
\IEEEauthorblockN{Rajni Kant Bansal, Pengcheng You, Dennice F. Gayme, and Enrique Mallada}
\IEEEauthorblockA{Whiting School of Engineering, Johns Hopkins University, Baltimore, MD, US\\
\{rbansal3, pcyou, dennice, mallada\}@jhu.edu}

\thanks{This work was supported by NSF through grants ECCS 1711188, CAREER ECCS 1752362, and CPS ECCS 2136324.}
}


\maketitle

\begin{abstract}

This paper proposes a market mechanism for multi-interval electricity markets with generator and storage participants. Drawing ideas from supply function bidding, we introduce a novel bid structure for storage participation that allows storage units to communicate their cost to the market using energy-cycling functions that map prices to cycle depths. The resulting market-clearing process--implemented via convex programming--yields corresponding schedules and payments based on traditional energy prices for power supply and per-cycle prices for storage utilization. We illustrate the benefits of our solution by comparing the competitive equilibrium of the resulting mechanism to that of an alternative solution that uses prosumer-based bids. Our solution shows several advantages over the prosumer-based approach. It does not require a priori price estimation. It also incentivizes participants to reveal their truthful cost, thus leading to an efficient, competitive equilibrium. Numerical experiments using New York Independent System Operator (NYISO) data validate our findings.

\end{abstract}


\section{Introduction}

Energy storage systems like lithium-ion batteries have the technical capability to provide essential grid services for system reliability and power quality. These capabilities combined with the growing adoption of non-dispatchable renewable energy sources are driving growing participation of energy storage in grid operation and electricity markets~\cite{aeo2021,no2222,denholm2019potential}.
A number of market dispatch models utilizing storage have been proposed for the purposes of e.g., integrating renewable energy sources~\cite{solar_integr,wind_integ}, supporting transmission and distribution networks~\cite{trans_distr_supp,freq_reserve_req}, providing demand response~\cite{demand_resp}. However, most of these models assume the canonical market mechanism for economic dispatch, which was designed without accounting for the operational cost of storage, that does not depend on energy supply, but rather charging-discharging cycles.

Recent works have sought to account for storage usage cost in the grid dispatch in two ways. 
The first approach seeks to develop sequences of charge-discharge bids or control actions, using existing market and reserve interfaces so as to maximize the storage operator revenue (market payments minus storage operation cost)~\cite{he2015optimal,XuTPSfactoring,shi2017optimal,lin_degrd_flex_bid}.
\textcolor{black}{The resulting optimization strategies rely on unknown prices that must be estimated~\cite{he2015optimal,XuTPSfactoring,shi2017optimal}, or accounted in the worst case \cite{lin_degrd_flex_bid}}.
The second approach incorporates the cost of storage operation by explicitly introducing a usage cost based on either energy cycles in the dispatch problem~\cite{GHe_marginal,bansal2020storage} or other proxies for storage degradation~\cite{linear_cost_coeff_bid}. These strategies usually assume storage owners to be truthful in revealing their  cost and do not endow them with the flexibility of seeking profit maximization. Moreover, these works provide little insight into how storage owner incentives can affect the ability of a system operator to efficiently operate the grid.

This paper provides insight into this problem through a novel market mechanism design that captures the effect of storage and generator owner incentives. We compare this approach to existing strategies through analysis of the overall system efficiency via competitive equilibrium characterization. More specifically, we consider a multi-interval market model where generators are endowed with a quadratic cost and bid using a supply function, while storage owners quantify the storage usage cost based on the degradation induced by the energy cycles. 
Our formulation exploits previous work combining the Rainflow cycle counting algorithm with a cycle stress function to obtain a \emph{notably convex} cycling cost function~\cite{shi2017optimal,bansal2020storage}.

We consider two different bidding strategies for storage. In the first setting, storage bids as a prosumer using a generalized supply function~\cite{yue_prosumer}, that allows it to behave as supply and demand, and is compensated based on spot prices. Although such a market achieves a competitive equilibrium, it requires that storage owners have a priori knowledge of cleared prices, and leads to prices and dispatch schedules that do not minimize the social cost.  
In order to overcome this inability to minimize social costs, we propose a new mechanism where storage owners bid using an energy-cycling function. This function maps prices (in dollars per cycle depth) to the corresponding cycle depth that the user is willing to perform, and allows storage participants to be compensated based on a per-cycle basis. We show that by properly adapting the market-clearing to account for this bid, the competitive equilibrium of this mechanism exists, and leads to a dispatch that minimizes the overall social cost. These goals are achieved by inducing a truthful bidding among storage owners that is independent of the clearing prices.
Numerical simulations show the advantages of the proposed cycle based market mechanism by evaluating the social cost and storage profit of the two mechanisms. We also include a setting in which storage cost is fully disregarded in the market-clearing as a baseline.

The rest of the paper is organized as follows. In Section~\ref{sec_2} we introduce the social planner problem and storage cost model. In Section~\ref{sec_3} we characterize the equilibrium in the prosumer based market where all participants participate using (generalized) linear supply function bids and compare it with the social planner optimal solution. The energy-cycling based function for storage 
is discussed in Section~\ref{sec_4}. We provide the numerical illustrations and conclusions in Section~\ref{sec_5} and Section~\ref{sec_6}, respectively.

\section{Social Planner Problem}\label{sec_2}

In this section we formulate a social planner problem that aims to achieve the optimal economic dispatch by minimizing the total cost of dispatching both generators and storage units.

\subsection{Problem Formulation}

Consider a multi-interval horizon $ \{1,2,...,T\}$ where a set $\mathcal{G}$ of generators and a set $\mathcal{S}$ of storage units participate in a market to meet a given inelastic demand profile $d\in\mathbb{R}^T$. 
For each generator $j \in \mathcal{G}$, the power output over the time horizon is denoted by a vector $g_j \in  \mathbb{R}^T$ whose elements are each subject to capacity constraints 
\begin{equation}\label{gen_limit_ineq}
    \underline{g}_j \le g_{j,t} \le \overline{g}_j, \quad t\in \{1,2,\hdots,T\},
\end{equation}
where $\underline{g}_j, \overline{g}_j$ denote the minimum and maximum generation limits, respectively. Analogously, for each storage unit $i \in \mathcal{S}$ of capacity $E_i$, the discharge (positive) or charge (negative) rates over the time horizon is denoted by a vector $u_i \in \mathbb{R}^T$. 
We assume each charge or discharge rate $u_{i,t}$ is bounded as
\begin{equation}
    \underline{u}_i \leq u_{i,t} \leq \overline{u}_i, \quad t\in \{1,2,\hdots,T\}. \label{SoC_control_limit}
\end{equation}
where $\underline{u}_i, \overline{u}_i$ denote the minimum and maximum rate limits, respectively. The corresponding amount of energy stored is characterized by a normalized State of Charge (SoC) profile $x_i \in \mathbb{R}^{T+1}$, with the initial SoC $x_{i,0} = x_{i,o}$. The SoC evolves over the time horizon according to
\begin{equation}
    x_{i,t} = x_{i,t-1} - \frac{1}{E_i}u_{i,t}, \ t \in \{1,2,...,T\}, \label{SoC_evolution_individual}.
\end{equation}
This evolution over the time horizon can be rewritten compactly as
\begin{equation}
    Ax_i = -\frac{1}{E_i}u_i ,\label{SoC_evolution}
\end{equation}
where
\begin{equation}
    A =  \begin{bmatrix}
-1 & 1 & 0 &\hdots &0\\
 0 &-1 & 1 &\ddots &\vdots\\
\vdots & \ddots &\ddots & \ddots & 0\\
0 & \hdots & 0 & -1 & 1\\
\end{bmatrix}\in \mathbb{R}^{T \times ({T+1})} .\nonumber 
\end{equation} 
In order to account for the cyclic nature of storage, we impose periodic constraints on the SoC, i.e.,
\begin{equation}
    x_{i,0} =  x_{i,T} = x_o. \label{SoC_terminal_const}
\end{equation}
Substituting~\eqref{SoC_evolution_individual} into~\eqref{SoC_terminal_const} leads to
\begin{equation}
    \mathbf{1}^Tu_{i} = 0.  \label{periodicity_const}
\end{equation} 
The normalized SoC satisfies
\begin{equation}
    0\leq x_{i,t} \leq 1 ,\quad t\in \{0,1,\hdots,T\} \label{Soc_limit_ineq},
\end{equation}
which can be rewritten using equation~\eqref{SoC_evolution_individual} and~\eqref{SoC_terminal_const} as
\begin{equation}
    (x_o - 1)\mathbf{1} \leq \Tilde{A}u_{i} \leq x_o\mathbf{1},  \label{Soc_limit_rate_ineq}
\end{equation}
where
\begin{equation}
    \Tilde{A} =  \frac{1}{E}\begin{bmatrix}
1 & 0  &\hdots &0\\
 1 &1  &\ddots &\vdots\\
\vdots  &\ddots & \ddots & 0\\
1 & \hdots  & 1 & 1\\
\end{bmatrix}\in \mathbb{R}^{T \times {T}}, \nonumber 
\end{equation} 
is a lower triangular matrix.
Finally, the social planner problem 
is given by 

\noindent \emph{SOCIAL PLANNER}
\begin{subequations}
\label{planner_problem}%
\begin{eqnarray}
    \min_{g_j,j\in\mathcal{G},u_i,i\in\mathcal{S}} &\!\!\!\sum_{i\in \mathcal{S}}C_i(u_i) +  \sum_{j\in \mathcal{G}}\left(\frac{c_j}{2} g_j^Tg_j + a_j \mathbf{1}^Tg_j \right)
    \label{planner_obj} \\
    \text{s.t.} 
    & d = \sum_{j\in\mathcal{G}}g_j +\sum_{i\in\mathcal{S}} u_i \label{power_bal_const} \\
    & \eqref{gen_limit_ineq},\eqref{SoC_control_limit}, \eqref{periodicity_const}, \eqref{Soc_limit_rate_ineq}, \nonumber
\end{eqnarray}
\end{subequations}
where \eqref{power_bal_const} enforces power balance for all time intervals. 
We use $C_i(u_i)$ to represent the operational cost of storage unit $i$, to be defined in the next subsection, and assume quadratic cost functions for the generators. 
For ease of analysis we assume without loss of generality that the linear coefficient $a_j = 0$. 

\subsection{Storage Cost Model}

The intrinsic degradation incurred due to repeated charging and discharging half-cycles\footnote{A full cycle is defined to consist of a charging half-cycle and a discharging half-cycle of the same depth.} constitutes the main operational cost of storage. We adopt the Rainflow cycle counting based method~\cite{lee2011rainflow,rainflow1971fatigue} to enumerate the cycles, which we incorporate into a cycle based cost function \cite{shi2017optimal,XuTPSfactoring,bansal2020storage}.
For each storage unit $i\in\mathcal{S}$, the Rainflow cycle counting algorithm maps the SoC profile $x_i$ to the associated charging-discharging half-cycles, summarized in a vector of half-cycle depths $\nu_i\in\mathbb{R}^T$, i.e.,
\[
    \nu_i := \Rainflow(x_i) .
\]
Using the cycle depth vector $\nu_i$ one can quantify the capacity degradation using a cycle stress function $\Phi(\cdot):[0,1]^T \mapsto [0,1]$.
In general, $\Phi(\cdot)$ is well approximated by a quadratic function~\cite{shi2017optimal}, thus we consider here
\[
    \Phi(\nu_i) := \frac{\rho}{2}\nu_i^T\nu_i ,
\]
where $\rho$ is a given constant coefficient~\cite{shi2017optimal,bansal2020storage}. This vector of identified half-cycle depths $\nu_i$ is used to compute the total degradation cost of storage $i$ as
\begin{equation}
      \frac{b_i}{2}\nu_i^T\nu_i\ =  \frac{\rho BE_i}{2}\nu_i^T\nu_i ,
    \nonumber
\end{equation}
where $B$ is the unit capital cost per kilowatt-hour of storage capacity $E_i$, and $b_i:= \rho B E_i $ is a constant. In order to define a storage cost function in terms of the storage charge-discharge rate vector $u_i$, we define a piece-wise linear mapping from this rate vector $u_i$ to the corresponding half-cycle depth vector $\nu_i$, as described in the following proposition.
\begin{proposition}
The total degradation cost $C_{i}(u_i)$ in \eqref{planner_obj} is given by 
\begin{equation}
    C_i(u_i) = \frac{b_i}{2}u_i^TN(u_i)^TN(u_i)u_i.
    \label{degradation_cost_rate}
\end{equation} 
The matrix $N(u_i)$ is defined as 
\begin{equation}
    N(u_i):= -\frac{1}{E_i}M(x_i)^TA^{\dagger} ,
\end{equation}
where the matrix $M(x_i)\in \mathbb{R}^{(T+1)\times T}$ is the incidence matrix for the SoC profile $x_i$~\cite{bansal2020storage} and satisfies
\begin{equation}
    \nu_i = \Rainflow(x_i) = N(u_i)u_i = M(x_i)^Tx_i.
    \label{depth_relation}
\end{equation}
\end{proposition}
\ifthenelse{\boolean{arxiv}}{
\begin{proof}
We can explicitly write the SoC profile $x_i$ from~\eqref{SoC_evolution} in terms of $u_i$ as:
\begin{equation}
    x_i = -\frac{1}{E_i}A^{\dagger}u_i + (I - A^{\dagger}A)w
\end{equation}
for any arbitrary vector $w \in \mathbb{R}^{T+1}$~\cite{james_1978}. Further from~\cite{bansal2020storage} we have the cost of degradation of storage as function of SoC profile $x_i$ explicitly as
\begin{equation}
    C_i(x_i) = \frac{\rho BE_i}{2}x_i^TM(x_i)M(x_i)^Tx_i 
    \label{degradation_cost_SoC}
\end{equation}
given a quadratic cycle stress function $\Phi$. Here the matrix $M(x_i)\in \mathbb{R}^{(T+1)\times T}$ is the incidence matrix for the SoC profile $x_i$ such that $M(x_i)^Tx_i = \nu_i$.

Given the matrix $A$ in~\eqref{SoC_evolution} and any incidence matrix $M(x_i)$, the following holds
\[
    M(x_i)^T(I-A^{\dagger}A) = \mathbf{0}_{T \times (T+1)} 
\]
To see this, notice that the columns of the incidence matrix $M(x_i)$ is given by a linear combination of the rows of the matrix $A$ and $A(I-A^{\dagger}A) = \mathbf{0}_{T \times (T+1)}$. In other words the matrix $M(x_i)^T$ is in the null space of matrix $I-A^{\dagger}A$. Therefore we have
\[
    M(x_i)^Tx_i = -\frac{1}{E_i}M(x_i)^TA^{\dagger}u = N(u_i)u_i
\]
such that $N(u_i):= -\frac{1}{E_i}M(x_i)^TA^{\dagger}$ and the total degradation cost is given by
\[
    C_i(u_i) = \frac{\rho BE_i}{2}u_i^TN(u_i)^TN(u_i)u_i = \frac{b_i}{2}u_i^TN(u_i)^TN(u_i)u_i.
\]
\end{proof}
}{The proof is provided in~\cite{bansal2021market}.}
For the example SoC profile in Fig.~\ref{fig:N_example}, the depth vector from the Rainflow cycle counting algorithm is $\nu=[x_1-x_2, x_1-x_2, x_3-x_0,x_3-x_4]^T$ and we obtain the associated matrix $N(u)$ such that $N(u)u = \frac{1}{E}[u_2,u_2,-u_1-u_2-u_3,u_4]^T = \nu$.
\begin{figure}[htp]
    \centering
    \includegraphics[width=8.5cm]{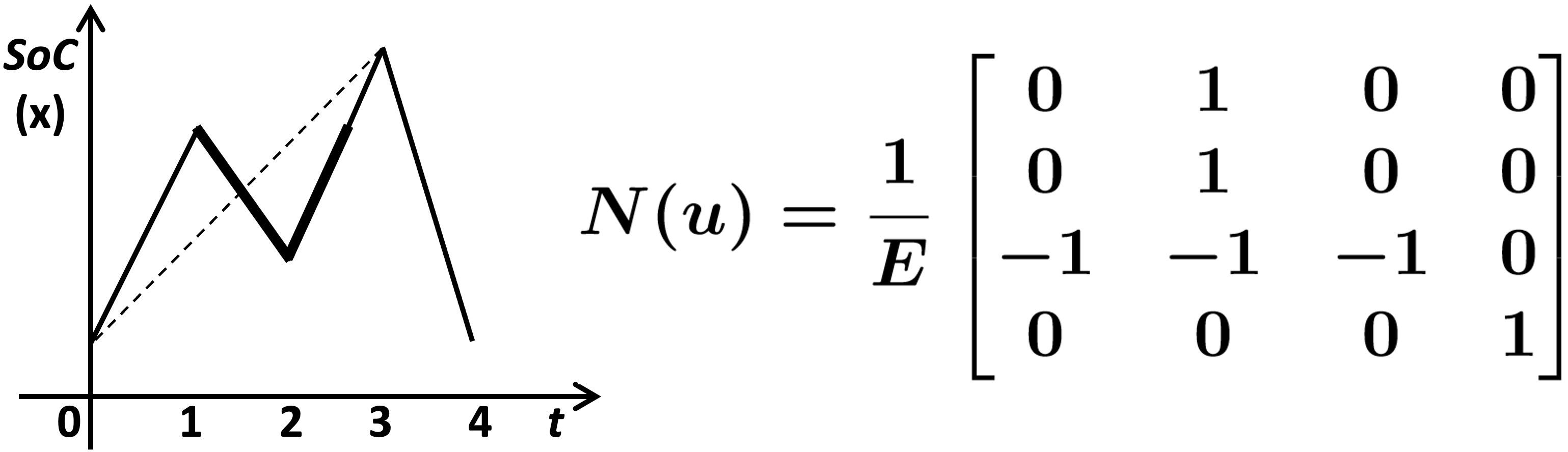}
    \caption{An example of SoC profile and associated $N(u)$ matrix.}
    \label{fig:N_example}%
\end{figure}%
\begin{remark} \label{multiple_Nmatrix}
    The piece-wise linear and temporally coupled cost function $C_i(\cdot)$ is convex~\cite{shi2017optimal,bansal2020storage}. However, the cost function~\eqref{degradation_cost_rate} is not differentiable everywhere with respect to the storage rate $u$ due to its piece-wise linear structure. At the point of non-differentiability we define all $m$ possible associated matrices for a given $u$ as $N_k(u), k\in\{1,2,...,m\}$ and the following relation holds: 
    \[
        N_k(u)u = N(u)u = \nu, \  \forall k \in \{1,2,...,m\}.
    \]
    See, e.g.,~\cite{bansal2020storage} for more details. 
\end{remark}

We illustrate this procedure for a non-differentiable profile in Fig.~\ref{fig:N_example_2} where $x_1=x_2$. Here the depth vector is $\nu=[x_1-x_2, x_1-x_2, x_3-x_0,x_3-x_4]^T$ and all the associated matrices $N_i(u) \ i\in\{1,2\}$ are given by $N_1(u)u = \frac{1}{E}[u_2,u_2,-u_1-u_2-u_3,u_4]^T = \nu$ and $N_2(u)u = \frac{1}{E}[0,0,-u_1-u_2-u_3,u_4]^T = \nu$. Here $N_1(u)$ and $N_2(u)$ represent the matrices associated with the SoC profile $x = [x_0,x_1\mp{\epsilon},x_2,x_3]$ for any $\epsilon \rightarrow 0^{+}$. 
\begin{figure}[htp]
    \centering
    \includegraphics[width=8.5cm]{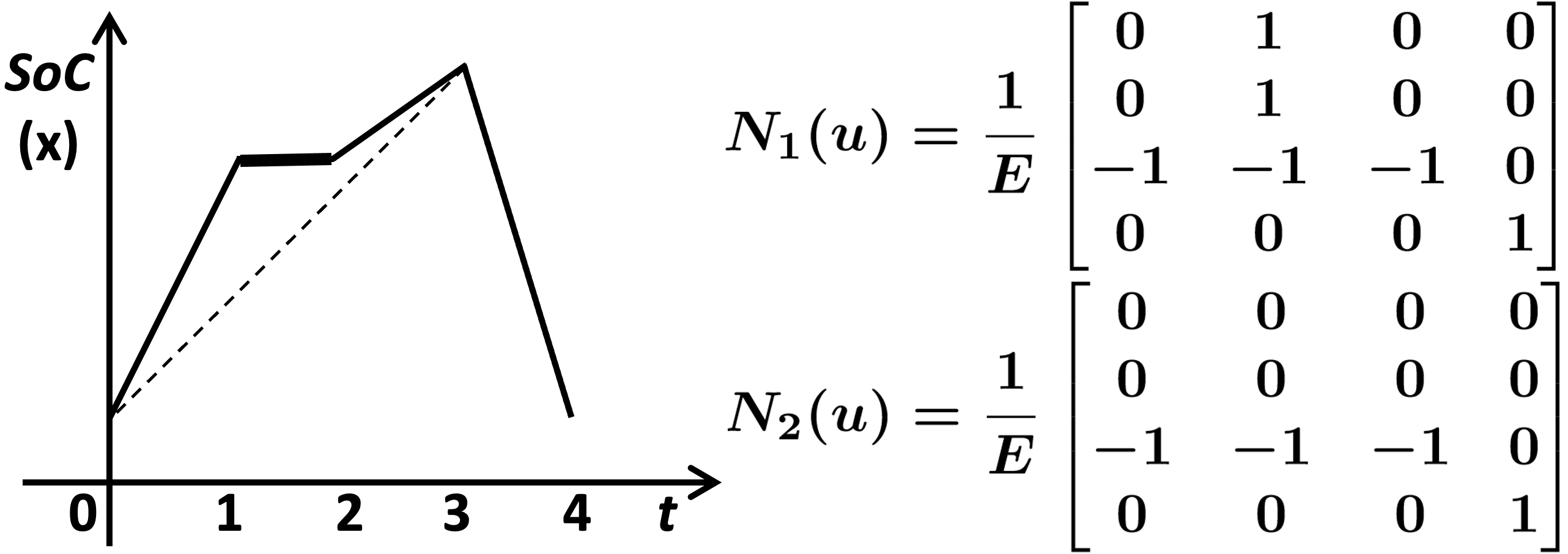}
    \caption{An example of an SoC profile and all the possible associated $N(u)$ matrices.}%
    \label{fig:N_example_2}%
\end{figure}%

\section{Prosumer Based Market Model} \label{sec_3}

We next exploit the analytical expression for the cost of storage degradation to evaluate whether the competitive equilibrium of participants bidding with the generalized linear supply function leads to an efficient system dispatch that minimizes the social cost. \textcolor{black}{
For ease of exposition we consider in this section the simplified setting where the only biding constraint in the market clearing is~\eqref{power_bal_const}. However, our results generalize beyond this assumption, at the cost of a more involved analysis.}
We next define the bidding form, the market-clearing, and the market settlement as part of the market mechanism.


\subsection{Market Mechanism}

We first formulate the market mechanism where all the participants submit a linear supply function. 
{Several approaches based on a linear supply function have been proposed to analyze participation of generators in the market~\cite{Joh2004Efficiency}. 
Here we extend this framework to  heterogeneous participants comprising both generators and storage.} We specify the bid for generator $j$ as
\begin{equation}
    g_j = \alpha_j \Theta_j \ 
    \label{generator_linear_supply}
\end{equation}  
and the bid for storage $i$ as
\begin{equation}
    u_i = \hat{\beta_i} \hat{\Theta}_i \ .
    \label{storage_linear_supply}
\end{equation} 
where $\Theta_j \in \mathbb{R}^T$ and $\hat{\Theta}_i\in \mathbb{R}^T$ denote the marginal prices aimed at incentivizing participation. These supply function bids are parameterized by $\alpha_j\ge 0$ and $\hat{\beta_i} \ge 0$, which indicate the willingness of generator $j$ and storage $i$ to produce at the price $\Theta_j$ and $\hat{\Theta}_i$, respectively. The market operator collects the supply function bids from all the participants and 
associates a cost function with generator $j$ that is given by
\[
    \sum_{t=1}^T\int_{0}^{g_{j,t}} \Theta_{j,t} \partial g_{j,t} = \sum_{t=1}^T\int_{0}^{g_{j,t}} \frac{1}{\alpha_j} g_{j,t}\partial g_{j,t}  = \frac{1}{2\alpha_j}g_j^Tg_j
\]
as well as a cost function for storage unit $i$ that is given by
\[
    \sum_{t=1}^T\int_{0}^{u_{i,t}} \hat{\Theta}_{i,t} \partial u_{i,t} = \frac{1}{2\hat{\beta_i}}u_i^Tu_i.
\]
Given the bids $(\alpha_j,j\in\mathcal{G}, \hat{\beta}_i i\in\mathcal{S})$, the operator solves the economic dispatch problem that minimizes total generation and storage costs to meet inelastic demand $d \in  \mathbb{R}^T$:

\noindent \emph{SYSTEM:}
\begin{subequations}
\begin{align}
    \min_{u_i, i\in \mathcal{S}, g_j,j \in \mathcal{G}} & \ \sum_{i \in \mathcal{S}}\frac{1}{2\hat{\beta_i}}u_i^Tu_i+\sum_{j \in \mathcal{G}}\frac{1}{2\alpha_j}g_j^Tg_j \\
    \textrm{s.t.} \ \ \ \ \ & 
    \eqref{power_bal_const}
\end{align}
\label{system_problem}
\end{subequations}
The optimal solution to the \emph{SYSTEM} gives the dispatch and the market-clearing prices for the participants at each time interval.
More precisely, the generator $j$ and the storage $i$ produce the dispatch quantities $g_j$ and $u_i$ and are paid $\Theta_j^Tg_j$ and $\hat{\Theta}_i^Tu_i$ as part of the market settlement respectively. 

{The individual prices $\Theta_j$ and $\hat{\Theta}_i$ are functions of dual variables associated with operational constraints of generators and storage. In the simplified setting where only the power balance constraint is binding $\Theta_j = \lambda$ and $\hat{\Theta}_i = \lambda$.} 

In this paper, we consider the price-taking behavior of the generators and the storage with characteristics summarized below. A price-taking assumption is usually evaluated as a benchmark in the sense that if a market mechanism does not behave as desired under price-taking assumptions, it is unlikely to perform well otherwise.
\begin{definition}
A market participant is price-taking if it accepts the given market prices and cannot influence the prices in the market on its own. 
\end{definition}

The participants choose their bids to maximize their individual profit as defined below:

\noindent \emph{Generator Bidding Problem}
\begin{subequations}
\begin{align}
    \max_{g_j} \ \pi(\lambda,g_j) := & \max_{g_j} \ \lambda^Tg_j - \frac{c_j}{2}g_j^Tg_j \label{generator_profit_problem} \\
    = & \max_{\alpha_j \ge 0} \ \alpha_j\lambda^T\lambda - \alpha_j^2\frac{c_j}{2}\lambda^T\lambda \label{generator_profit_problem_parameter}
\end{align}
\end{subequations}
\noindent \emph{Storage Bidding Problem}
\begin{subequations}
\begin{align}
    \max_{u_i} \ & \pi(\lambda,u_i) :=  \max_{u_i} \ \lambda^Tu_i - \frac{b_i}{2}u_i^TN(u_i)^TN(u_i)u_i
    \label{storage_profit_problem} \\
    = & \max_{\hat{\beta_i} \ge 0} \ \hat{\beta_i}\lambda^T\lambda - \frac{b_i\hat{\beta_i}^2}{2}\lambda^TN(\lambda,\hat{\beta_i})^TN(\lambda,\hat{\beta_i})\lambda
    \label{storage_profit_problem_parameter}
\end{align}
\end{subequations}
where we have substituted the linear supply function bids~\eqref{generator_linear_supply},\eqref{storage_linear_supply} with $\Theta_j = \hat{\Theta}_i = \lambda$ respectively.

\subsection{Market Equilibrium}

We next define and characterize the competitive equilibrium 
under which none of the participants has any incentive to change its decision while the market is cleared. 
\begin{definition}
Under price-taking assumptions we say the bids $(\hat{\beta_i},i\in \mathcal{S}, \alpha_j,j \in \mathcal{G}, \lambda)$ form a competitive equilibrium if the following conditions are satisfied:
\begin{enumerate}
    \item For each generator $j \in \mathcal{G}$, the bid $\alpha_j$ maximizes their individual profit in the market
    \item For each storage element $i \in \mathcal{S}$, the bid $\hat{\beta_i}$ maximizes their individual profit in the market
    \item The inelastic demand $d \in \mathbb{R}^T$ is satisfied with the market-clearing prices $\lambda$.
\end{enumerate}
\end{definition}
We first propose a lemma that will enable us 
to characterize the competitive equilibrium in the market.
\begin{lemma}
For any $\beta\in\mathbb{R}, \ \beta > 0$ and $\lambda \in \mathbb{R}^T$, the following holds
\[
    N(\beta\lambda) = N(\lambda).
\]
\label{N_indep_beta}
\end{lemma}

The proof uses the fact that the scalar multiplier $\beta$ only scales the input profile to the piece-wise linear map $N$ but the profile behaviour, i.e., charging-discharging characteristics remain unchanged. The following proposition characterizes the competitive equilibrium. 
\begin{theorem} \label{Thrm1}
The competitive equilibrium of the prosumer based market mechanism~\eqref{system_problem} is uniquely determined by:
\begin{subequations}
\begin{align}
    &\alpha_j = \frac{1}{c_j}, \ \forall j\in\mathcal{G} \\
    &\hat{\beta_i} = \frac{1}{b_i}\frac{\lambda^T\lambda}{\lambda^TN(\lambda)^TN(\lambda)\lambda}, \ \forall i\in\mathcal{S}\\
    &\lambda = \delta d, \  \delta^{-1}  = \left(\sum_{i\in \mathcal{S}}\frac{1}{b_i}\frac{d^Td}{d^TN(d)^TN(d)d}+\sum_{j\in \mathcal{G}}\frac{1}{c_j} \right)
\end{align}
\label{competitive_eqbm_traditional}
\end{subequations}
\end{theorem}

\ifthenelse{\boolean{arxiv}}{
\begin{proof}
Using Lemma~\ref{N_indep_beta} $N(\lambda,\hat{\beta_i}) = N(\lambda)$ and for ease of notation we denote $N:= N(\lambda)$. Now we solve~\eqref{storage_profit_problem_parameter} for optimal decision $\hat{\beta_i}$ as
\begin{align}
    & \frac{\partial \pi_{u_i}}{\partial \hat{\beta_i}} = \lambda^T\lambda - b\hat{\beta_i}\lambda^TN^TN\lambda - b\hat{\beta_i}\frac{\partial }{\partial \hat{\beta_i}}(\lambda^T N^TN\lambda) = 0 \nonumber\\
    & \implies \hat{\beta_i} = \frac{1}{b_i} \frac{\lambda^T\lambda}{\lambda^TN^TN\lambda}, \ \forall i \in \mathcal{S}. \label{opt_beta_hat}
\end{align}
since $N$ is independent of parameter $\hat{\beta}_i$. Similarly for the individual generator's decision parameter solving~\eqref{generator_profit_problem_parameter}
\begin{align}
    & \frac{\partial \pi_{g_j}}{\partial {\alpha_j}} = \lambda^T\lambda - \alpha_jc_j\lambda^T\lambda = 0  \implies {\alpha_j} = \frac{1}{c_j} \label{opt_alpha}, \ \forall j \in \mathcal{G}
\end{align}
For market-clearing prices $\lambda$, given the linear supply function bid~\eqref{generator_linear_supply},\eqref{storage_linear_supply}, the power balance constraint in~\eqref{power_bal_const} implies
\begin{subequations}
\begin{align}
    d = & \sum_{g_j,j\in\mathcal{G}}g_j +\sum_{u_i,i\in\mathcal{S}} u_i= \sum_{j\in\mathcal{G}}\alpha_j\lambda +\sum_{i\in\mathcal{S}}\hat{\beta_i}\lambda\\
    \implies \lambda = & \left(\sum_{i\in \mathcal{S}}\hat{\beta_i}+\sum_{j\in \mathcal{G}}\alpha_j \right)^{-1}d
\end{align}
\end{subequations}
Since $\lambda$ is proportional to $d$, let's assume $\exists$ $\delta\in\mathbb{R}$ such that $\lambda = \delta d$. Using~\eqref{opt_alpha} and \eqref{opt_beta_hat} we have a unique $\delta$ as
\[
    \delta^{-1} = \left(\sum_{i\in \mathcal{S}}\hat{\beta_i}+\sum_{j\in \mathcal{G}}\alpha_j \right) = \left(\sum_{i\in \mathcal{S}}\frac{1}{b_i} \frac{d^Td}{d^TN^TNd}+\sum_{j\in \mathcal{G}}\frac{1}{c_j} \right)
\]
Hence, under price-taking assumption the tuple $(\hat{\beta_i}, i\in\mathcal{S},\alpha_j,j\in\mathcal{G},\lambda)$ uniquely exists. 
\end{proof}
}{The proof is provided in~\cite{bansal2021market}.} Though the market achieves unique competitive equilibrium, the optimal decision parameter of the storage is a temporally coupled function of the market-clearing prices unknown to the participants beforehand.  Moreover, the equilibrium achieved requires restrictive conditions on the market to align its solution with the social planner's optimum as characterized in the next subsection. In particular, the linear supply function bidding mechanism fails to reflect the true cost of storage participation even in the simple market setting. 
 
\subsection{Social Welfare Misalignment}


We next characterize the gap between the market equilibrium and social optimum in the following theorem.

\begin{theorem}
    The competitive equilibrium $(g_j^{*},\ j\in \mathcal{G}, u_i^{*}, \ i\in\mathcal{S}, \lambda^{*})$ for the \textit{SYSTEM}~\eqref{system_problem} solves the \textit{SOCIAL PLANNER}~\eqref{planner_problem} if and only if there exists convex coefficients $\gamma_k \ge 0, \ \sum_{k=1}^m\gamma_k =1$,  such that the following holds
    \begin{align}
        \sum_{k=1}^m\gamma_kN_k^T(d)N_k(d)d = \frac{d^TN(d)^TN(d)d}{d^Td}d.
        \label{comp_social_solve_cond}
    \end{align}
\end{theorem}
\begin{proof}
    Since the social planner problem~\eqref{planner_problem} is convex and all the constraints are affine, linear constraint qualification is satisfied and the associated KKT conditions are both sufficient and necessary. Given that $(g_j^{*},\ j\in \mathcal{G}, u_i^{*}, \ i\in\mathcal{S}, \lambda^{*})$ also solves the social planner problem, the set of the solution must satisfy the KKT conditions given by:
    \begin{subequations}
    \begin{align}
        & c_jg^{*}_j = \lambda^{*}, \ \forall j\in\mathcal{G} \label{gen_marginal_cond}\\
        & \sum_{k=1}^m\gamma_kb_iN_k(u_i^{*})^TN_k(u_i^{*})u^{*}_i = \lambda^{*}, \ \forall i\in\mathcal{S}. \label{storage_marginal_cond}
    \end{align}
    \end{subequations}
along with the primal feasibility constraints given by~\eqref{power_bal_const}. Here $\gamma_k, \ k\in \{1,2,...,m\}$ are the convex coefficient such that~\eqref{storage_marginal_cond} also holds at the non-differentiable optimal solution $u_i^{*}$. Since $u_i^{*} = \hat{\beta_i}\lambda^{*}, \forall i\in\mathcal{S}$ we rewrite~\eqref{storage_marginal_cond} as 
\begin{align}
        &\lambda^{*} =  \sum_{k=1}^m\gamma_kb_iN_k(\lambda^{*})^TN_k(\lambda^{*})u^{*}_i, \ \forall i\in\mathcal{S}. \label{storage_marginal_cond_2}
\end{align}

Also from the competitive equilibrium in Theorem~\ref{Thrm1} we have 
\begin{align}
    & \lambda^{*} = \frac{1}{\hat{\beta}_i}u_i^{*} = b_i\frac{\lambda^{{*}^T}N(\lambda^{*})^TN(\lambda^{*})\lambda^{*}}{\lambda^{{*}^T}\lambda^{*}}u_i^{*}, \ \forall i\in\mathcal{S} \label{thrm_comp_cond}
\end{align}
Here $N(\lambda^{*})u_i^{*} = N_k(\lambda^{*})u_i^{*}$ for any $k \in\{1,2,...,m\}$, recall Remark~\ref{multiple_Nmatrix}. Now combining  equations~\eqref{storage_marginal_cond_2} and~\eqref{thrm_comp_cond} we have the following relation $\forall i\in\mathcal{S}$
\begin{align}
    & \sum_{k=1}^m\gamma_kb_iN_k(\lambda^{*})^TN_k(\lambda^{*})u^{*}_i = b_i\frac{\lambda^{{*}^T}N(\lambda^{*})^TN(\lambda^{*})\lambda^{*}}{\lambda^{{*}^T}\lambda^{*}}u_i^{*} \nonumber \\
    & \iff  \sum_{k=1}^m\gamma_kN_k(\lambda^{*})^TN_k(\lambda^{*})\lambda^{*} = \frac{\lambda^{{*}^T}N(\lambda^{*})^TN(\lambda^{*})\lambda^{*}}{\lambda^{{*}^T}\lambda^{*}}\lambda^{*} \nonumber \\
    & \iff \sum_{k=1}^m\gamma_kN_k(d)^TN_k(d)d = \frac{d^TN(d)^TN(d)d}{d^Td}d, \nonumber
\end{align}
where the second last equality holds due to $u_i^{*} = \hat{\beta}_i\lambda^{*}$ and the last equality holds due to the relation $\lambda^{*} = \delta d$ from~\eqref{competitive_eqbm_traditional}. 
\end{proof}

While the linear supply function bidding mechanism does reflect the quadratic cost function for generators\footnote{ For the generator cost function in~\eqref{planner_obj} with $c_j\neq 0, a_j \neq 0$, a time dependent bid $\alpha_j(t) \in \mathbb{R}^{T}$ can be used to reflect a general cost function.}, in general it fails to reflect the incentive of storage unit in the market as the associated condition~\eqref{comp_social_solve_cond} may not hold. As an example where this condition holds, consider a market with 2 time periods $t =\{1,2\}$ and the inelastic demand $d = d_0[1,1]^T$ such that $d_0\in\mathbb{R}_{+}$. Since $u \propto d$ at the market equilibrium, the only associated matrix $N = \big(\begin{smallmatrix}
  1 & 1\\
  0 & 0
\end{smallmatrix}\big)$,  such that
\begin{align}
    & N(d)^TN(d)d = \begin{pmatrix} 1&1\\1&1\end{pmatrix}d = \begin{pmatrix} 2d_0\\2d_0\end{pmatrix} \nonumber \\
    & \frac{d^TN(d)^TN(d)d}{d^Td}d = \frac{4d_0^2}{2d_0^2}d = N(d)^TN(d)d. \nonumber
\end{align}

This misalignment between the market equilibrium and the social optimum motivates our mechanism design in the next section.

\section{Cycle Aware Market Model} \label{sec_4}

In this section we propose a new market mechanism that incentivizes both generators and storage to bid in a manner that reflects their true cost under price-taking assumptions. 

\subsection{Market Mechanism}

We consider generators that provide the linear supply function bids defined in~\eqref{generator_linear_supply} and generalize this idea to propose an energy-cycling function bid for storage. In particular, the price-taking storage $i$ indicates the schedule of cycle depths as function of per-cycle prices given by
\begin{equation}
    \nu_i = \beta_i \theta_i \ .\label{slot_agnostic_supply_function} 
\end{equation}
Here $\theta_i \in \mathbb{R}^{T}$ are per-cycle prices aimed at incentivizing storage participation.
This function is parameterized by a constant $\beta_i \ge 0$ and indicates all the charging or discharging half-cycle depths $\nu_i$ the storage is willing to undergo at the price $\theta_i$.  With per-cycle prices $\theta_i$ from the market, the storage unit $i$ can choose its bid in order to maximize its profit, which as function of the bid $\beta_i$, is given by:
\begin{align}
    \pi_{u_i}(\beta_i,\theta_i) = & \theta_i^T\nu_i - \frac{b_i}{2}u_i^TN(u_i)^TN(u_i)u_i= \theta_i^T\nu_i - \frac{b_i}{2}\nu_i^T\nu_i \nonumber\\
    = & \beta_i\theta_i^T\theta_i-\frac{b_i\beta_i^2}{2}\theta_i^T\theta_i, \label{str_profit_slot_agnostic_bid}
\end{align}
where we used~\eqref{slot_agnostic_supply_function} and \eqref{depth_relation} to obtain~\eqref{str_profit_slot_agnostic_bid}. A price-taking storage owner seeks to maximize~\eqref{str_profit_slot_agnostic_bid}, i.e., find $\beta_i$ that satisfies:
\begin{align}\label{cycle_aware_opt_beta}
        \frac{\partial \pi_{u_i}}{\partial \beta_i} = \left(1-b_i\beta_i\right))\theta_i^T\theta_i = 0 \implies \beta_i = \frac{1}{b_i}, \ \forall i\in\mathcal{S}
\end{align}
Thus, this cycle aware market mechanism leads to an optimal bid $\beta_i$ that is not only independent of prices in the market but also truthful.

We now illustrate a cycle aware market-clearing where both generator and storage are incentive compatible and market aligns with social planner's problem while satisfying the demand. In this setting, the market operator collects supply function bids from all the participants and solves the following economic dispatch problem that minimizes the total cost of generator and storage dispatch:

\noindent \emph{CYCLE AWARE SYSTEM:}
\begin{subequations}
\begin{align}
    \min_{(u_i,\nu_i), i\in \mathcal{S}, g_j,j \in \mathcal{G}} & \ \sum_{i \in \mathcal{S}}\frac{1}{2\beta_i}v_i^Tv_i+\sum_{j \in \mathcal{G}}\frac{1}{2\alpha_j}g_j^Tg_j \label{mixed_market_obj}\\
    \textrm{s.t.} \ \ \ \ \ & v_i = N(u_i)u_i, \ i\in \mathcal{S} \label{rainflow_const}\\
    & 
    \eqref{gen_limit_ineq},\eqref{SoC_control_limit}, \eqref{periodicity_const}, \eqref{Soc_limit_rate_ineq}, \eqref{power_bal_const}\nonumber%
\end{align}\label{mixed_market_dispatch}%
\end{subequations}%
where~\eqref{rainflow_const} implements the Rainflow algorithm. The optimal solution to cycle aware system \eqref{mixed_market_dispatch} gives the dispatch and two set of prices. More precisely as part of the market settlement, generator $j$ produces $g_j$ and gets paid $\Theta_j^Tg_j$ where 
\[
    \Theta_j = \lambda + \underline{\eta}_j-\overline{\eta}_j. 
\]
Here $\underline{\eta}_j, \overline{\eta}_j$ are the dual variables associated with the generator capacity constraint~\eqref{gen_limit_ineq} and the vector $\lambda$ denotes the market-clearing prices or the dual variable associated with the constraint~\eqref{power_bal_const}. The storage unit $i$ produces a cycle depth schedule $\nu_i$ and gets paid $\theta_i^T\nu_i$ with the prices $\theta_i$ given by the dual variable associated with the constraint~\eqref{rainflow_const}. 

The piece-wise linear constraint~\eqref{rainflow_const} makes the dispatch problem~\eqref{mixed_market_dispatch} non-convex and challenging to solve numerically. However, substituting the rainflow constraint~\eqref{rainflow_const} in the cost function objective~\eqref{mixed_market_obj} leads to an equivalent convex optimization problem that can be solved by the convex programming as formalized in the following proposition.
\begin{proposition}\label{optimality_lemma}
        Any locally optimal solution $(g_j,j\in\mathcal{G}, u_i, i\in\mathcal{S}, \nu_i, i\in\mathcal{S}, \lambda, \theta_i, i\in\mathcal{S})$ of the cycle aware system~\eqref{mixed_market_dispatch} is also a globally optimal solution.
\end{proposition}
\ifthenelse{\boolean{arxiv}}{
\begin{proof}
    For the dispatch problem~\eqref{mixed_market_dispatch} denote the dual variable associated with the constraint~\eqref{periodicity_const},\eqref{power_bal_const},\eqref{rainflow_const} as $\delta_i, i\in\mathcal{S}$, $\lambda$, and $\theta_i, i\in\mathcal{S}$ respectively. Further define $(\underline{\eta}_j,\overline{\eta}_j), j \in \mathcal{G}$, $ (\overline{\mu}_i,\underline{\mu}_i), i\in\mathcal{S}, \mbox{ and }(\underline{\nu}_i,\overline{\nu}_i), i\in\mathcal{S}$ to be the non-negative dual variables associated with the inequality constraints~\eqref{gen_limit_ineq}, \eqref{SoC_control_limit},\eqref{Soc_limit_rate_ineq}, respectively. The necessary KKT conditions require the stationarity condition: 
    \begin{subequations}\label{cycle_aware_market_kkt}
    \begin{align}
        &\lambda^{*} = \frac{g_j^{*}}{\alpha_j}+\overline{\eta}_j^{*}-\underline{\eta}_j^{*}, \ \forall j\in\mathcal{G} \label{cycle_system_kkt_1}  \\
        &\theta_i^{*} = \frac{\nu_i^{*}}{\beta_i}, \ \forall i\in\mathcal{S} \label{cycle_system_kkt_2} \\
        & \lambda^{*} =\sum_{k=1}^m\hat{\gamma_k}N_k(u_i^{*})^T\theta_i^{*} + \delta_i^{*}\mathbf{1}+\Tilde{A}^T(\overline{\mu}_i^{*}-\underline{\mu}_i^{*}) \nonumber \\
        & \ \ \ \ \ \ \ \  +\overline{\nu}_i^{*}-\underline{\nu}_i^{*}, \ \forall i\in\mathcal{S} \label{cycle_system_kkt_3}
    \end{align}
    \end{subequations}
    Here $\hat{\gamma_k} \ge 0, \ \sum_{k=1}^m\hat{\gamma_k}=1$ denote the convex coefficients associated with $m$ possible $N$ matrices. The complimentary slackness require $\forall j \in \mathcal{G}$:
    \begin{align}
        &{\overline{\eta}_j^{*}}^T(g_j^{*}-\overline{g}_j\mathbf{1}) = 0 \ &{\underline{\eta}_j^{*}}^T(\underline{g}_j\mathbf{1}-g^{*}_j) = 0  \label{cycle_aware_market_kkt_gen_comp}
    \end{align}
    and $\forall i \in \mathcal{S}$:
    \begin{subequations}\label{cycle_aware_market_kkt_str_comp}
    \begin{align}
        & {\overline{\mu}_i^{*}}^T(\Tilde{A}u^{*}_i-x_o\mathbf{1}) = 0 \ &{\underline{\mu}_i^{*}}^T((x_o-1)\mathbf{1}-\Tilde{A}u^{*}_i) = 0 \label{cycle_system_kkt_4}\\
        &{\overline{\nu}_i^{*}}^T(u_i^{*}-\overline{u}_i\mathbf{1}) = 0 \ &{\underline{\nu}_i^{*}}^T(\underline{u}_i\mathbf{1}-u^{*}_i) = 0  \label{cycle_system_kkt_5}
    \end{align}
    \end{subequations}
    Furthermore the primal feasibility are given by the constraints~\eqref{gen_limit_ineq},\eqref{SoC_control_limit},\eqref{periodicity_const},\eqref{Soc_limit_rate_ineq},\eqref{power_bal_const},\eqref{rainflow_const} while the dual feasibility requires non-negativity of the dual variables associated with the inequality constraints~\eqref{gen_limit_ineq},\eqref{SoC_control_limit} and \eqref{Soc_limit_rate_ineq}.  
    Now we rewrite the dispatch problem~\eqref{mixed_market_dispatch} and then use method of contradiction to prove the statement. The dispatch problem~\eqref{mixed_market_dispatch} can be rewritten as below:
    \begin{align}\label{convex_dispatch}
         \min_{u_i,i\in\mathcal{S},g_j,j\in\mathcal{G}} \ & \sum_{j\in\mathcal{G}}\frac{1}{2\alpha_j}g_j^Tg_j+\sum_{i\in\mathcal{S}}\frac{1}{2\beta_i}u_i^TN(u_i)^TN(u_i)u_i\\
        \textrm{s.t.} \ &  \eqref{gen_limit_ineq},\eqref{SoC_control_limit}, \eqref{periodicity_const}, \eqref{Soc_limit_rate_ineq}, \eqref{power_bal_const} \nonumber \nonumber
    \end{align}
    where we substitute the rainflow constraint~\eqref{rainflow_const} in the objective~\eqref{mixed_market_obj}. Note that~\eqref{convex_dispatch} is convex~\cite{shi2017optimal, bansal2020storage}. Here we abuse the notation and denote the dual variable associated with the constraint~\eqref{periodicity_const}, \eqref{power_bal_const}, and \eqref{rainflow_const} as $\delta_i, i\in\mathcal{S}$, and $\lambda$ respectively. Also define $(\underline{\eta}_j,\overline{\eta}_j), j \in \mathcal{G}$, $ (\overline{\mu}_i,\underline{\mu}_i), i\in\mathcal{S}, \mbox{ and }(\underline{\nu}_i,\overline{\nu}_i), i\in\mathcal{S}$ to be the non-negative dual variables associated with the inequality constraints~\eqref{gen_limit_ineq},\eqref{SoC_control_limit},\eqref{Soc_limit_rate_ineq}, respectively. The necessary and sufficient KKT conditions require the stationarity:  
    \begin{subequations}
    \label{convex_dispatch_kkt}
    \begin{align}
        &\lambda^{*} = \frac{g_j^{*}}{\alpha_j}+\overline{\eta}_j^{*}-\underline{\eta}_j^{*}, \ \forall j\in\mathcal{G} \label{convex_dispatch_kkt_1}  \\
        &\theta_i^{*} = \frac{\nu_i^{*}}{\beta_i}, \ \forall i\in\mathcal{S} \label{convex_dispatch_kkt_2} \\
        & \lambda^{*} =\frac{1}{\beta_i}\sum_{k=1}^m\gamma_kN_k(u_i^{*})^TN_k(u_i^{*})u_i^{*} + \delta_i^{*}\mathbf{1}+\Tilde{A}^T(\overline{\mu}_i^{*}-\underline{\mu}_i^{*}) \nonumber \\
        & \ \ \ \ \ \ \ \  +\overline{\nu}_i^{*}-\underline{\nu}_i^{*}, \ \forall i\in\mathcal{S} 
         \label{convex_dispatch_kkt_3}
    \end{align}
    \end{subequations}
    where $\gamma_k \ge 0, \ \sum_{k=1}^m\gamma_k=1$ denote the convex coefficients associated with $m$ possible $N$ matrices. Similarly the complimentary slackness is given by~\eqref{cycle_aware_market_kkt_gen_comp},\eqref{cycle_aware_market_kkt_str_comp} while the primal feasibility are given by the constraints~\eqref{gen_limit_ineq},\eqref{SoC_control_limit},\eqref{periodicity_const},\eqref{Soc_limit_rate_ineq},\eqref{power_bal_const},\eqref{rainflow_const}. And the dual feasibility requires non-negativity of the dual variables associated with the inequality constraints~\eqref{gen_limit_ineq},\eqref{SoC_control_limit} and \eqref{Soc_limit_rate_ineq}.   
    
    For any optimal solution of~\eqref{convex_dispatch} given by $(g_j^{*},j\in\mathcal{G}, u_i^{*}, i\in\mathcal{S})$, $\exists$ solution $(g_j^{*},j\in\mathcal{G}, u_i^{*}, i\in\mathcal{S}, \nu_i^{*}, i\in\mathcal{S})$ that also satisfies the KKT conditions~\eqref{cycle_aware_market_kkt} where $\nu_i^{*} = N(u_i^{*})u_i^{*}, \forall i\in\mathcal{S}$. Therefore it is locally optimal solution of the disptach problem~\eqref{mixed_market_dispatch}. WLOG assume any optimal solution given by $(\hat{g}_j,j\in\mathcal{G}, \hat{u}_i, i\in\mathcal{S}, \hat{\nu}_i, i\in\mathcal{S})$ with strictly smaller cost to the dispatch problem~\eqref{mixed_market_dispatch}. Since $(\hat{g}_j,j\in\mathcal{G}, \hat{u}_i, i\in\mathcal{S})$ also satisfies the KKT condition of the convex problem~\eqref{convex_dispatch}, therefore it is also an optimal solution with strictly smaller cost which is a contradiction. Therefore any locally optimal solution given by $(g_j^{*},j\in\mathcal{G}, u_i^{*}, i\in\mathcal{S}, \nu_i^{*}, i\in\mathcal{S}, \lambda^{*}, \theta_i^{*}, i\in\mathcal{S})$ is also a globally optimal solution.
\end{proof}
}{The proof is provided in~\cite{bansal2021market}.} \textcolor{black}{The equivalent convex problem formulation ensures that the optimal dispatch and clearing prices are incentive compatible~\cite{bansal2020storage}. We discuss the competitive equilibrium in such a market in the next subsection.}

\subsection{Market Equilibrium}

We next redefine and characterize the competitive equilibrium of the market competition under the proposed mechanism.
\begin{definition}
Under the price-taking assumptions, the bids $(\beta_i,i\in \mathcal{S}, \alpha_j,j \in \mathcal{G})$ form a competitive equilibrium if the following conditions are satisfied:
\begin{enumerate}
    \item For each generator $j \in \mathcal{G}$, the bid $\alpha_j$ maximizes their individual profit in the market
    \item For each storage element $i \in \mathcal{S}$, the bid $\beta_i$ maximizes their individual profit in the market
    \item For each storage element $i \in \mathcal{S}$, the Rainflow constraint is satisfied with per-cycle prices $\theta_i$.
    \item The inelastic demand $d \in \mathbb{R}^T$ is satisfied with the market-clearing prices $\lambda$.
\end{enumerate}
\end{definition}

The following is our main result and highlights the alignment of proposed market mechanism with the social planner's problem.

\begin{theorem}\label{optimal_theorem}
    The competitive equilibrium of the cycle aware market mechanism~\eqref{mixed_market_dispatch} also solves the SOCIAL PLANNER problem~\eqref{planner_problem}. 
\end{theorem}

\begin{proof}
    Under price-taking assumptions, given prices $(\Theta_j,j\in\mathcal{G})$ from the market-clearing~\eqref{mixed_market_dispatch}, the optimal bid is given by 
    \begin{align}
        & \frac{\partial \pi_{g_j}}{\partial {\alpha_j}} = \frac{\partial}{\partial {\alpha_j}}\left({\Theta_j}^Tg_j - \frac{c_j}{2}g_j^Tg_j\right) = 0  \implies {\alpha_j}^{*} = \frac{1}{c_j}. \nonumber 
    \end{align}
     and for $(\theta_i,i\in\mathcal{S})$ the optimal decision parameter from~\eqref{cycle_aware_opt_beta} is given by 
    \[
        \beta_i^{*} = \frac{1}{b_i},\forall i\in\mathcal{S}.
    \]
    Therefore using the optimal solution of~\eqref{mixed_market_dispatch} from Proposition~\ref{optimality_lemma} along with the the optimal bid of participants given by $(\beta_i^{*}, i\in\mathcal{S}, \alpha_j^{*},j\in\mathcal{G})$, we recover the social planner problem~\eqref{planner_problem}. Hence the competitive equilibrium of~\eqref{mixed_market_dispatch} also solves the \emph{SOCIAL PLANNER} problem. 
\end{proof}

We end by noting that, although the cycle aware system in~\eqref{mixed_market_dispatch} may have non-unique optimal schedule $u_i^{*}$ and $\theta_i^{*}$ due to the piece-wise linear rainflow constraint~\eqref{rainflow_const}, the cycle aware market mechanism aligns with the social planner problem, and any such solution will be optimal. \textcolor{black}{We further note that the additional assumption of uniform pricing leads to the set of unique optimal schedule and prices. This case is not discussed further here, due to page limits.}

\section{Numerical Simulation} \label{sec_5}

In this section we provide a numerical example comparing the competitive equilibrium of the Prosumer based market mechanism \emph{(PBM)} and the Cycle based market mechanism \emph{(CBM)}. We use aggregate demand data of the Millwood Zone operated by the NYISO (date: 8/10/2020)~\cite{nyisodata}. For ease of analysis we assume one generator and one aggregate storage unit. The aggregate cost coefficients of the generator are $c = 0.1 \$/(MW)^2$ and $a = 20\$/MW$ in equation \eqref{planner_obj}~\cite{matpower} and the empirical cost coefficients of the quadratic cycle stress function is $\rho = 5.24 \times 10^{-4}$~\cite{shi2017optimal}. The generation has sufficient capacity to meet the demand, i.e. $\underline{g} = 0$ and $\overline{g} \ge \max_t\{d_t\}$. The storage rate limits are given by 
$\overline{u} = \frac{E}{4}$ and $\underline{u} = -\frac{E}{4}$, which corresponds to storage requiring four hours (slots) to completely charge or discharge. 

We use a canonical Generation Centric Dispatch (\emph{GCD}) model in which market accounts for only generation cost i.e. disregarding storage degradation cost from the objective function leading to a cycle unaware dispatch strategy as a benchmark case.
This \emph{hidden} cycling cost is calculated from the storage SoC profile of the optimal solution. These costs are then added to the cost function value to compute the total social cost, i.e. social cost = generation cost + (\emph{hidden}) cycling cost.

\begin{figure}[ht]
    \centering
    \subfloat{{\includegraphics[width=4.4cm]{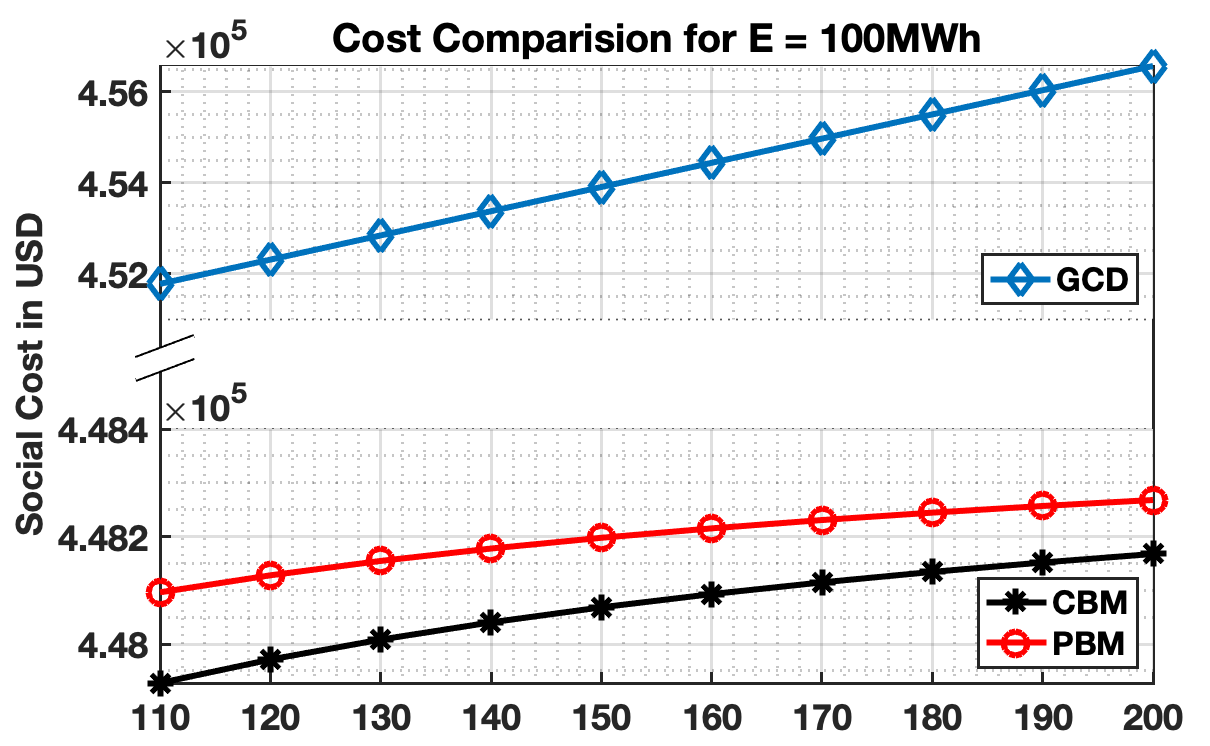} }}
    \subfloat{{\includegraphics[width=4.4cm]{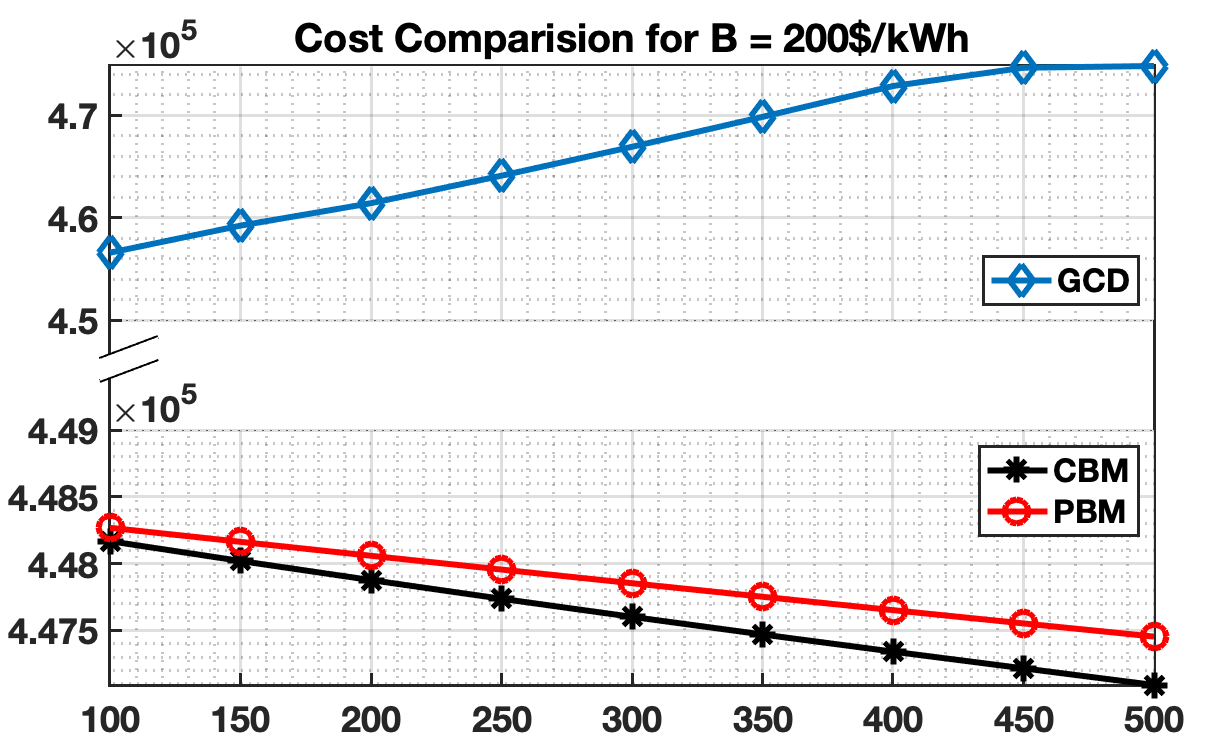} }}
    \newline
    \centering
    \setcounter{subfigure}{0}
    \subfloat[\label{fig:cost_B}\centering Cost w.r.t storage capital cost]{{\includegraphics[width=4.4cm]{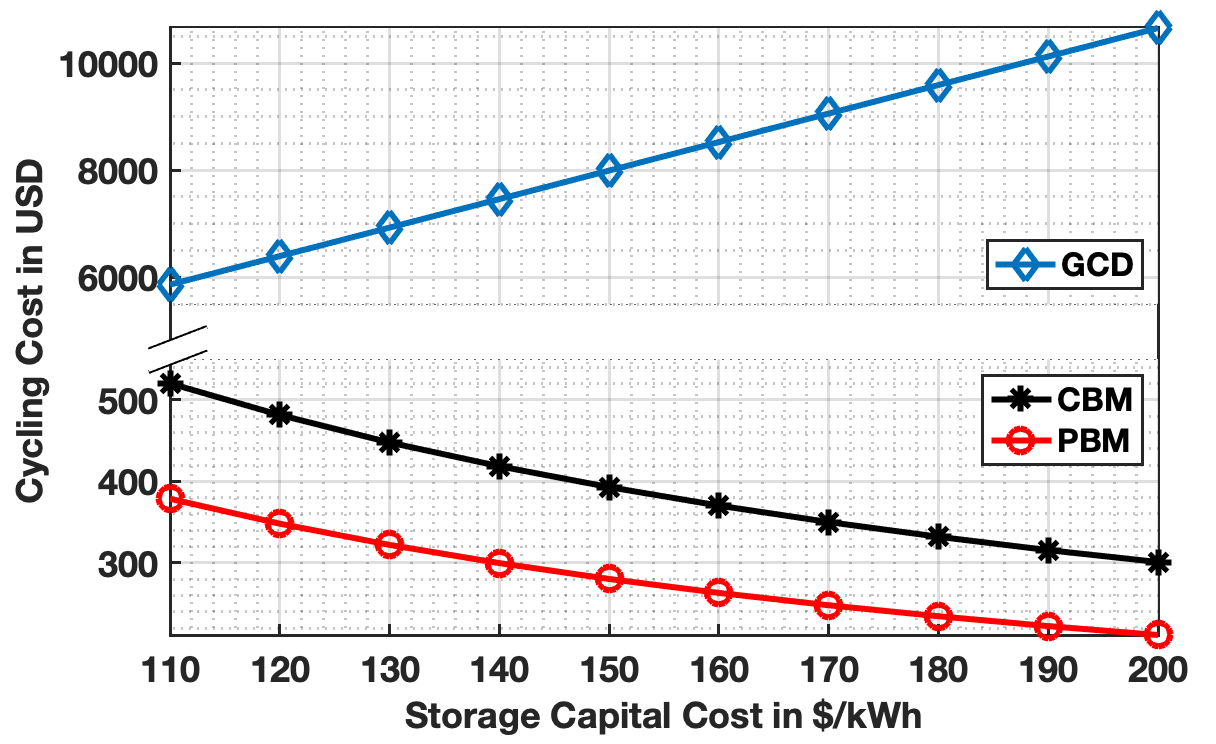} }}
    \subfloat[\label{fig:cost_E}\centering Cost w.r.t storage capacity]{{\includegraphics[width=4.35cm]{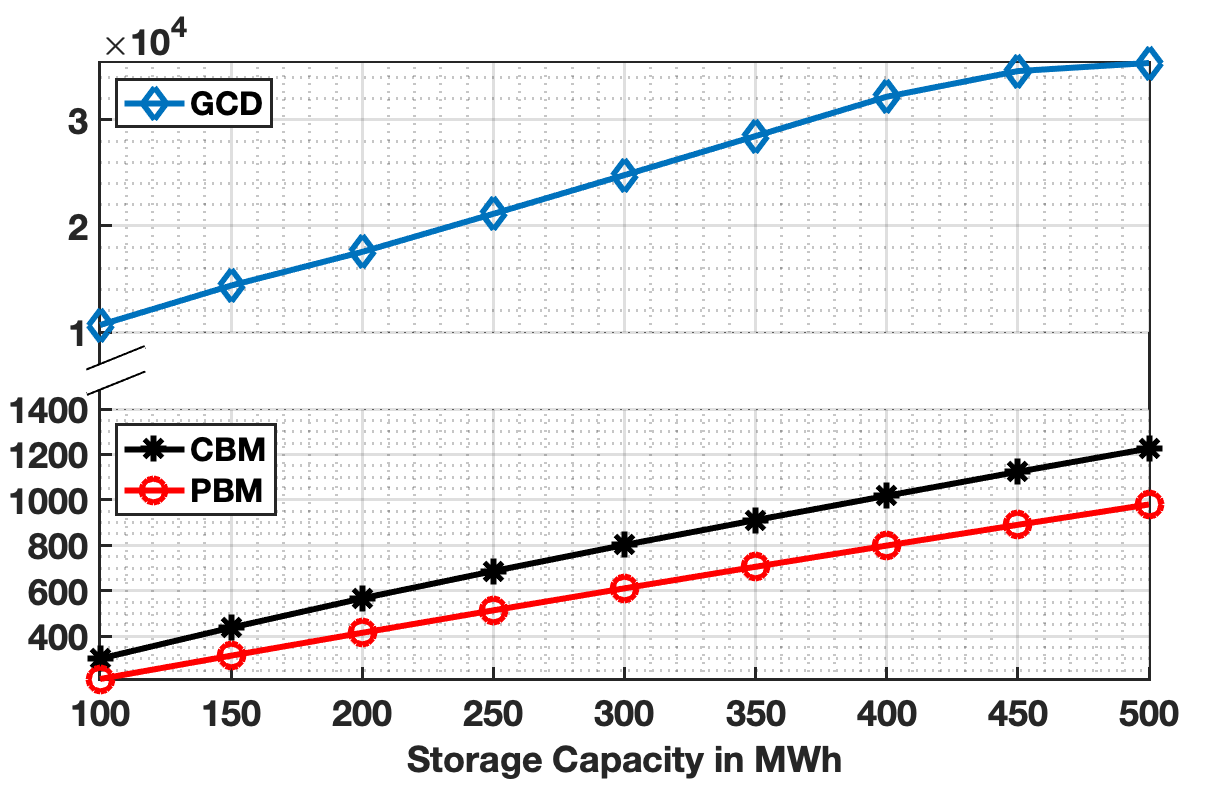} }}
    \caption{(Top) Social cost and (Bottom) Cycling cost in the Cycle based mechanism (CBM), Prosumer based mechanism (PBM) and Generation centric dispatch (GCD) w.r.t storage capital cost and storage capacity.}%
    \label{fig:cost}%
\end{figure}

Fig.~\ref{fig:cost} illustrates the social cost and cycling cost of storage as we (a) increase the storage capital cost given a fixed storage capacity and (b) increase the storage capacity given a fixed storage capital cost. 
In the top panel of  Fig.~\ref{fig:cost_B} we fix the storage capacity to be $E = 100 MWh$. As expected the social cost increases with the capital cost and our proposed CBM gives the lowest social cost, while GCD has the highest costs as it does not account for cycling costs in the optimization. The bottom panel in Fig.~\ref{fig:cost_B} shows the cycling cost of storage under the three mechanisms. As expected GCD utilizes the storage without any restrictions leading to higher cost. Since the PBM overestimates the cost of storage, it leads to more restrictive use of storage and hence lower cycling cost compared to CBM. Although CBM incurs higher storage cost, the incentive compatibility allows it to reduce the total social cost. 

In Fig.~\ref{fig:cost_E} we fix the storage capital cost to be $B =200 \$/kWh$~\cite{mongird2019energy} and increase storage capacity.
The social cost decreases with the storage capacity for CBM and PBM. Thus the benefits of accounting for degradation increase when storage capacity increases (top panel of Fig.~\ref{fig:cost_E}). Further, not accounting for storage cost, as in GCD, leads to overall higher social cost.
This is because as the capacity increases, storage can supply the required power with fewer or (relatively) shallower cycles, thus decreasing the social cost of CBM and PBM. The cycling cost is shown in the bottom panel in Fig.~\ref{fig:cost_E}. 

\begin{figure}[ht]
    \centering
    \subfloat[\label{fig:str_profit_B}\centering Profit w.r.t storage capital cost]{{\includegraphics[width=4.4cm]{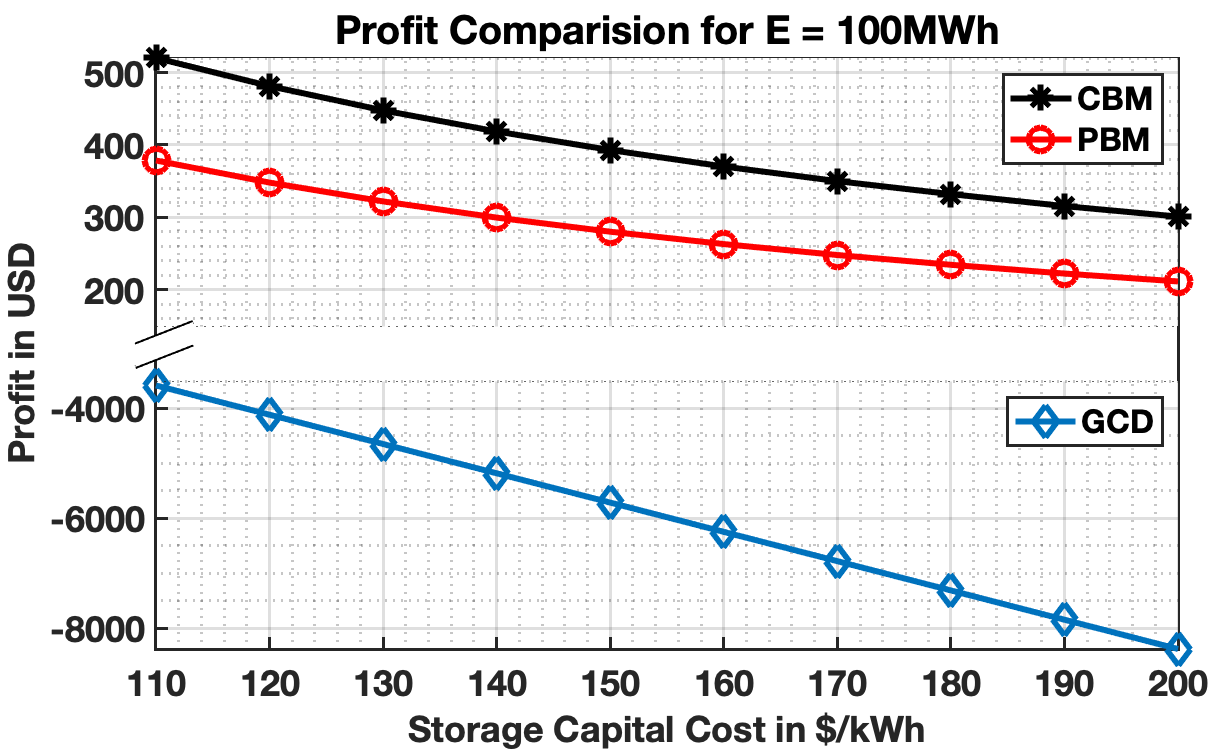} }}
    \subfloat[\label{fig:str_profit_E}\centering Profit w.r.t storage capacity]{{\includegraphics[width=4.5cm]{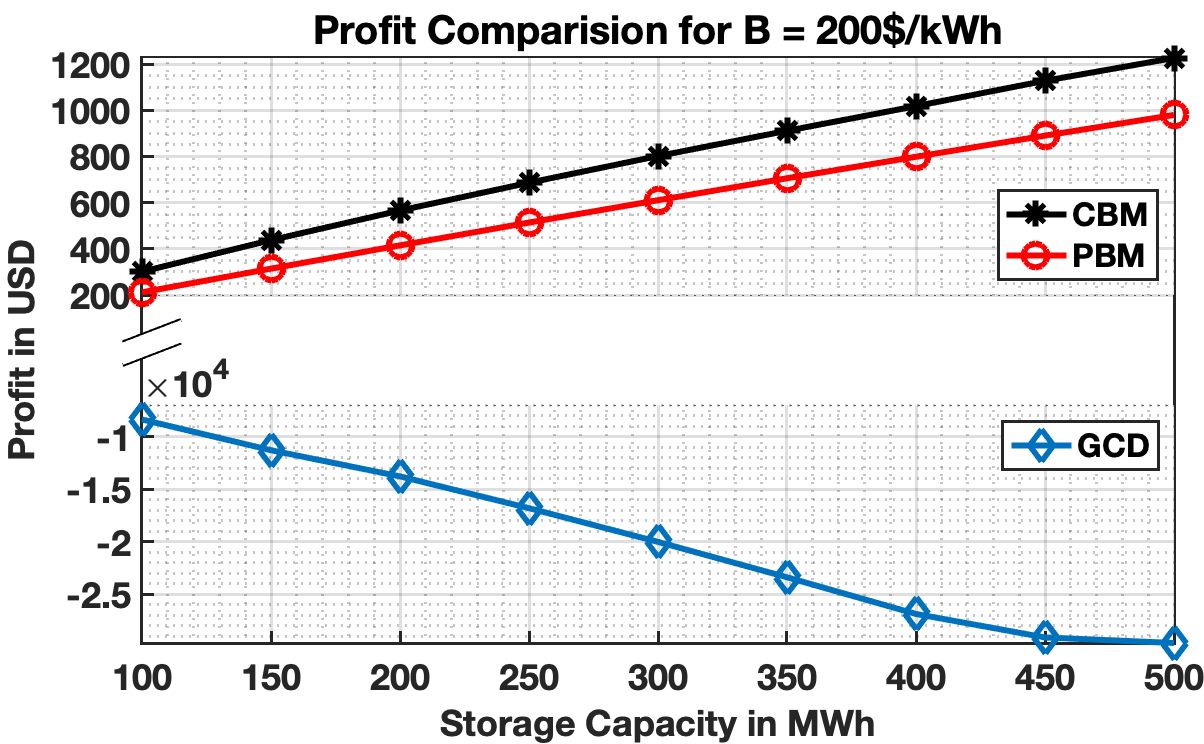} }}
    \caption{Storage profit in the Cycle based mechanism (CBM), Prosumer based mechanism (PBM) and Generation centric dispatch (GCD) w.r.t storage capital cost and storage capacity.}%
    \label{fig:profit_diff}%
\end{figure}

Fig.~\ref{fig:profit_diff} compares the profit of the storage as we increase the storage capital cost and increase the storage capacity respectively. In Fig.~\ref{fig:str_profit_B} we fix the storage capacity to be $E =100 MWh$ whereas in Fig.~\ref{fig:str_profit_E} we fix storage capital cost to be $B = 200\$/kWh$. As expected storage earns more profit in CBM as compared to PBM due to realistic representation of cost of storage degradation. The GCD leads to losses for the storage due to the large unaccounted for cycling cost of storage. As storage capital cost decreases or the storage capacity increases, storage incurs lower degradation cost while earning more profit at the peak period.

\section{Conclusions} \label{sec_6}

In this paper, we study the incentive of generators and storage units via market equilibrium analysis. We first analyze a prosumer based market mechanism where both generators and storage bid linear supply functions. Under the price-taking assumptions, the competitive equilibrium in such a market requires restrictive conditions to align with the social optimum. Furthermore, the optimal bid of storage is a temporally coupled function of market prices which is not desirable.
To address these shortcomings, we propose a novel energy-cycling function for storage where storage bids cycle depths as a function of per-cycle prices. This type of storage bidding function along with a linear supply function for generators incentivizes the participants to reflect their true cost in the market i.e. the competitive equilibrium in such a market minimizes the social cost. Numerical examples illustrate the importance of conveying storage owner incentives to market operators to achieve social welfare.

\bibliographystyle{IEEEtran}
\bibliography{PSCC2022}

\end{document}

%% file: edits_pcyou.tex
\usepackage{ifthen}
\newboolean{showcomments}
\setboolean{showcomments}{true}
\usepackage[usenames,dvipsnames]{xcolor}
\usepackage{soul}

\definecolor{bleudefrance}{rgb}{0.19, 0.55, 0.91}
\definecolor{ao(english)}{rgb}{0.0, 0.5, 0.0}

\newcommand{\addcite}[0]{\ifthenelse{\boolean{showcomments}}
{\textcolor{purple}{(add cite(s)) }}{}}%

\newcommand{\you}[1]{  \ifthenelse{\boolean{showcomments}}
{\todo[inline,color=bleudefrance]{PCY: #1}}{}}
\newcommand{\youmargin}[1]{\ifthenelse{\boolean{showcomments}}{\marginpar{\color{bleudefrance}\tiny PCY: #1}}{}}
\newcommand{\youumargin}[1]{  \ifthenelse{\boolean{showcomments}}
{\todo[color=bleudefrance,size=\tiny]{PCY: #1}} {}}

\newcommand{\pcyou}[1]{\ifthenelse{\boolean{showcomments}}
	{ \textcolor{red}{(PCY:  #1)}}{}}

\newboolean{showedits}
\setboolean{showedits}{true}
\usepackage[xcolor=bleudefrance]{changes}
\definechangesauthor[color=bleudefrance]{PCY}
\newcommand{\ayou}[1]{
\ifthenelse{\boolean{showedits}}
{\added[id=PCY]{#1}}
{#1}
}
\newcommand{\chyou}[2]{
\ifthenelse{\boolean{showedits}}
{\replaced[id=PCY]{#1}{#2}}
{#1}
}
\newcommand{\dyou}[1]{
\ifthenelse{\boolean{showedits}}
{\deleted[id=PCY]{#1}}
{}
}

%% file: edits.tex
\usepackage{ifthen}
\setboolean{showcomments}{true}

\newcommand{\enrique}[1]{  \ifthenelse{\boolean{showcomments}}
{\todo[inline,color=bleudefrance]{Enrique: #1}}{}}
\newcommand{\emmargin}[1]{\ifthenelse{\boolean{showcomments}}{\marginpar{\color{bleudefrance}\tiny EM: #1}}{}}

\setboolean{showedits}{false}
\definechangesauthor[color=bleudefrance]{EM}
\newcommand{\aem}[1]{
\ifthenelse{\boolean{showedits}}
{\added[id=EM]{#1}}
{\!#1\hspace{-4.75pt}}
}
\newcommand{\repem}[2]{
\ifthenelse{\boolean{showedits}}
{\replaced[id=EM]{#1}{#2}}
{\!#1\hspace{-4.75pt}}
}
\newcommand{\dem}[1]{
\ifthenelse{\boolean{showedits}}
{\deleted[id=EM]{#1}}
{}
}